\documentclass[a4paper, 11pt, leqno, twoside]{amsart}
\usepackage{hyperref}
\usepackage{natbib,amsmath,amssymb,latexsym,amsthm,enumerate,epsfig,graphicx, color}

\newtheorem{definition}{Definition}[section]

\newtheorem{proposition}{Proposition}[section]
\newtheorem{lemma}{Lemma}[section]

\newtheorem{remark}{Remark}[section]
\numberwithin{equation}{section}

\newcommand{\ep}{\varepsilon}
\newcommand{\disp}{\displaystyle}

\newcommand{\ccal}{{\mathcal C}}

\newcommand{\ecal}{{\mathcal E}}

\newcommand{\Ical}{{\mathcal I}}

\newcommand{\jcal}{{\mathcal J}}

\newif\ifcomment \commentfalse
\def\commentON{\commenttrue}

\long\outer\def\BC#1\\EC{\ifcomment \sloppy \par \# \ldots\dotfill
{\em #1} \dotfill \# \par \fi } \commentON

\newcommand{\remove}[1]{}

\newcommand{\R}{\mathbb{R}}

\renewcommand{\d}{\mathop{{\textsc{d}}}}
\newcommand{\s}{\mathop{\textsc{s}}}

\renewcommand{\c}{\mathop{\textsc{c}}}

\begin{document} 
\title{Rare Mutations in Evolutionary Dynamics} 
\author[A.L.~Amadori, A.~Calzolari,  R.~Natalini, B.~Torti]{Anna Lisa Amadori$^{1}$, Antonella Calzolari$^{2}$, \\ Roberto Natalini$^{3}$, Barbara Torti$^{2}$}
\address{$^1$Dipartimento di Scienze Applicate, Universit\`a di Napoli \lq\lq Parthenope''}
\address{$^2$ Dipartimento di Matematica, Universit\`a di Roma \lq\lq Tor Vergata''}
\address{$^3$ Istituto per le Applicazioni del Calcolo \lq\lq M. Picone'', Consiglio Nazionale delle Ricerche}

\maketitle

\section*{Abstract} In this paper we study the effect of rare mutations, driven by a marked point process, on the evolutionary behavior of a population. We derive a Kolmogorov equation describing the expected values of the different frequencies and  prove some rigorous analytical results about their behavior. Finally, in a simple case of two different quasispecies, we are able to prove that  the rarity of mutations increases the survival opportunity of the low fitness species. 

\medskip

\subparagraph*{\textup{2000} Mathematics Subject Classification:} 92D25 (35R09, 60G55)

\medskip

\subparagraph*{Keywords:} Coevolutionary dynamics, mutations, marked point processes, partial integro-differential equations.

\section{Introduction}
Evolutionary dynamics describes biological systems in terms of three general principles: replication, selection and mutation.
Each biological type -- a genome or a phenotype as well as a species -- is described by its reproduction rate, or fitness. In force of selection, a population evolves and changes its fitness landscape. Genetic changes can help in reaching some local optimum, or open a path to a new fitness peak, but sometimes they may drift population away from a peak, especially if the mutation rate is high. See  \cite{Nowakbook} for an extensive account of the state of the art concerning evolutionary dynamics.

A mixed population of constant size constituted by a fixed number of different types is characterized by the vector collecting the relative abundance of each type: 
$x=(x_0, x_1,\dots,x_{\d})$. By virtue of evolution, this  vector draws a path in the  simplex 
\[ S=\left\{ x=(x_0,x_1\dots x_{\d})\in\R^{\d+1} \, : \, \sum\limits_{k=0}^{\d} x_k= 1, \,  x_k\ge 0 \text{ as } k=0,1\dots\d\right\}.\]
The mechanism of  replication/selection is well described by an ordinary differential equation, where the relative fitness measures the balance  between death and birth of individuals. Denoting by $f_k$ the absolute fitness of any  $k$-type, by $f=(f_0, x_1,\dots,f_{\d})$ the  fitness vector, and by $\disp\bar{f}=  x\cdot f$ the mean fitness of the population, this equation reads
\begin{equation}\label{replicator}
\frac{d x_k }{dt}= \left(f_k - \bar{f}\right) x_k  , \qquad \qquad \text{ as } k=0,1,\dots \d.
\end{equation}
Several shapes have been proposed for the absolute fitness. When one is modelling phenotypes, the choice of a constant fitness seems fair, but, starting with the seminal work by  \cite{MSP73},  an important amount  of research  deals with ideas arising from mathematical game theory, see \cite{SHbook} and references therein. In that framework, the individual fitness is taken as a linear function of the population $x$, i.e.~$f(x) = A x$,  where $\disp A$  is the payoff matrix that rules  the interplay between different strategists. 
In this case, equation  \eqref{replicator} is the celebrated replicator equation introduced by \cite{TJ1978}.

Concerning mutations,  it has to be mentioned the quasispecies equation introduced by \cite{ESbook79}, where constant fitness were considered. This choice modifies equation \eqref{replicator} into 
\begin{equation}\label{quasispecies}
\frac{d x_k}{dt}  = \sum\limits_{i=0}^{\d}  f_i  \, q_{ik}\, x_i -   \bar{f} x_k , \qquad \qquad \text{ as } k=0,1,\dots \d.
\end{equation}
Here the coefficient $q_{i k}$ express the proportion of offspring of $k$-type from a progenitor $i$, which shows up at any  procreation. It is clear that $q_{i k}=\delta_{i k}$ gives back the equation \eqref{replicator}.
When the fitness vector is given by the relation $f(x) = A x$, as suggested by evolutionary game theory, then equation \eqref{quasispecies} is the well-known replicator-mutator equation, also known as selection-mutation equation, studied in \cite{SS92mutation}. 

As a matter of fact, mutations introduce a random ingredient into evolution, that is not enough emphasized in \eqref{quasispecies}. 
\cite{TCH06} pointed out that equation \eqref{quasispecies} can be recovered by assuming that the population follows a generalized Moran process and taking the limit for large population size. \cite{CFM08} and \cite{JMW12levyflights} showed that  various macroscopic diffusion models can be derived by the same individual stochastic process, by performing different types of rescaling. We also mention \cite{DL96}, where a macroscopic dynamics is deduced by  an individual based  stochastic description of the mutation process.

Here, we prefer to take a more macroscopic viewpoint, which however takes strongly into account the different regime of mutation processes. We start by modelling the stochastic dynamics at the level of the frequency vector $x$.  In order to capture the ``rarity'' of mutations, we assume that they are driven by some point  processes.
As a result, the random path of the frequency vector $x$ in the simplex is not continuous. This happens because mutations arise at a different time-scale with respect to replication and selection, and this assumption constitutes the main novelty of the present paper.

The stochastic  dynamics for frequencies is introduced and studied  in Section \ref{sec2}.
In Section \ref{sec3}, a Kolmogorov equation describing the expected values is rigorously derived and studied in its analytical aspects. Because of the point process, such equation is of integro-differential type. Global existence of mild solutions to this equation is proved as well as some useful basic estimates. 
Next, in Section \ref{sec4}, we deal with the particular case of two different quasispecies. We prove the existence of a stable equilibrium, that in general is greater or equal to the one of the standard quasispecies equation, to point out the fact that the rarity of mutations increases the survival opportunity of the low fitness species. In this sense, the single relevant contribution of our paper is to show that  the equilibrium position depends not only on the global amounts of mutations, but also on the time intensity of the process driving mutations. As the time intensity goes to infinity, the standard quasispecies equilibrium is recovered. But when mutation are concentrated in few, very rare events, a different equilibrium arises. We prove this occurrence in a rigorous way in Section \ref{sec4}, using some refined uniform estimates of the derivatives of the solution, in the simple case of two different quasispecies, in a given range of the parameters of the model. 
 
\section{The stochastic model}\label{sec2}

We introduce here a stochastic differential equation that describe the population distribution, under  rare sudden mutations.
Let us first revise the selection-mutation equation \eqref{quasispecies}. 
To underline the effect of mutation, we  follow \cite{SHbook} and introduce the ``effective mutation matrix'' $M= Q-I$.
If no mutation occurs, $M$ is the null matrix; in general $\disp M=(m_{ik})_{i,k=0,1,\dots,\d}$ with $ m_{ik} =q_{ik}\in[0,1]$ if $i\neq k$ and $m_{kk}= q_{kk}-1=-\sum\limits_{i\neq k} m_{ki}\in [-1,0]$.
Equation \eqref{quasispecies} can be written as
\begin{align} \label{quasispecies1} 
\frac{d}{dt} x_k = &  \left(f_k(x) - \bar{f}(x)\right) x_k  +\sum\limits_{i=0}^{\d} f_i(x) \, m_{ik} \,  x_i  ,
\end{align}
for $t>0$ as $k=0,1,\dots \d$. Here and henceforth we have assumed that all $f_k(x)$ are nonnegative.
We remark that the term $\left(f_k(x)- \bar{f}(x)\right) x_k$ stands for the homogeneous reproduction steered by the replicator equation, 
 the term $\sum\limits_{i\neq k} f_i(x) \, m_{ik} \, x_i \ge 0 $ describes the increasing of the frequency of $k$-type yielded by birth of $k$-individuals by mutation, and the term $  f_k(x) \, m_{kk}  \, x_k \le 0$ measures the decreasing of the frequency of $k$ individuals caused by the birth of mutated descendants by $k$-type progenitors.
The underlying assumption is that  mutations happen at the same time-scale as homogeneous reproduction: at any  procreation, a fixed proportion of the progenies  shows up a mutant trait.
It seems however more realistic to describe mutations as sudden changes in the population distribution: we thus assume here that they are driven by some marked point  processes.  
To be more precise, let $(T_n,Z_n)_n$ be a sequence of random times and random marks on a filtered probability space $(\Omega,\mathcal{F},(\mathcal{F}_t),{\mathrm P})$.
The marks are chosen in the mark space $Z=\{(i,k)\in \{0,\ldots,\d\}^2,\;i\neq k\}$, i.e.~it is assumed that each mutation has progenitor of a fixed type and descendants of a different unique type.
Set
\[ N_t=\sum_n\mathbb{I}_{(T_n\leq t)},\;\;\;\;\;\;N^{ik}_t=\sum_n\mathbb{I}_{(Z_n=(i,k))}\mathbb{I}_{(T_n\leq t)},\;\;i\neq k, \phantom{bla} t\ge 0,\]
which define the processes that count respectively the total number of mutations and the number of mutations with ancestor $i$ and descendants $k$, so that $N_t=\sum_{i\neq k}N^{ik}_t$.
The intensity of $N^{ik}_t$ should depend on the ``genetic distance'' from type $i$ to type $k$. Besides, it is affected also by selection: the larger the fitness $f_i({ x_{t^-}})$, the more  type $i$ reproduces, the more often its offspring shall suffer a mutation.
Hence we take as  $(\mathcal{F}_t)$-intensity kernel 
\[ \lambda_t(dz)=\sum_{i\neq k}\lambda_{ik}f_i(x_{t^-})\mathbb{I}_{(i,k)}(dz), \] 
where $\lambda_{ik}$ are positive constants. It follows that, for each $i\neq k$, $N^{ik}$ is a point process with $(\mathcal{F}_t)$-intensity equal to $\lambda_{ik}f_i(x_{t^-})$.  We remark that when the fitness vector $f$ is constant then, for each $i\neq k$, $N^{ik}$ is a classical Poisson process of parameter $\lambda_{ik}f^i$ (see also Remark \ref{confrontoquasispecie}).

The proportion of the offspring of $i$-individuals showing a $k$-type by effect of mutation is taken constant and denoted by $\gamma_{ik}\in (0,1]$.
A stochastic dynamics for the relative frequencies arises
\begin{align} \label{quasispeciesstoch}
d x^k_t = &  \left(f_k(x_t) - \bar{f}(x_t)\right) x^k_t \, dt   + \sum\limits_{i\neq k} \gamma_{ik} \,  x^i_t d N^{ik}_t -   \sum\limits_{i\neq k} \gamma_{ki} \,  x^k_t d N^{ki}_t,
\end{align}
as $k=0,1\dots\d$.
We next address to well-posedness of the S.D.E.~\eqref{quasispeciesstoch}. In view of writing its infinitesimal generator (and, later on, its Kolmogorov equation), we introduce  the vector valued function $a(x) =\left(a_0(x),a_1(x)\dots a_{\d}(x)\right)$, with
\begin{align}\label{a}
a_k(x)= & - \left(f_k(x) - \bar{f}(x)\right) x_k ,  \qquad k=0,1\dots\d ,\\
\intertext{and  the first order discrete non-local functional} \label{J}
 {\mathcal J} \phi(x)& =  \sum\limits_{i\neq j}\lambda_{ij} f_i(x) \left[ \phi\left(x+ \gamma_{ij}x_i ({e_j}-{e_i})\right)- \phi(x) \right] .
\end{align}
Here $e_j$ stands for the unit vector pointing in the direction of the $j^{th}$ axis.
\remove{The corresponding dynamic for the frequencies vector is
\begin{align}\label{vect-freq-eq-integral-1}
d{ x_t}= -a(x_t)\,dt+\sum_{i\neq k}
\gamma_{ik}x^i_{t^-}(e_{k}- e_{i})\,dN^{ik}_t\end{align}
where $a(x)=\big(a_0(x),a_1(x),\ldots,a_{\d}(x)\big)$ is the vector valued function
\begin{align*}a_j(x)=-\big(f_j({x})- \overline{f}({x})\big) x_j,\;\;\;j=0,1,\ldots,\d.\end{align*}}
Taking into account that for $(i,k)\neq (j,l)$ the processes
$N^{ik}$ and $N^{jl}$ have no common jump times, Ito's formula for
semimartingales gives
\begin{align}\label{mg-eq} d\phi({x}_t)= \left(-a({x}_{t}) \cdot \phi({x}_{t})+ {\mathcal J} \phi(x_{t^-}) \right) dt+ dM_t\end{align}
 for each function $\phi\in\mathcal{C}^1_b(\mathbb{R}^{\d+1})$. Here $M$ is the $(\mathcal{F}_t)$-martingale defined by
\[ dM_t=\sum_{i\neq k}\left[\phi({x}_{t^-}+ \gamma_{ik}x^i_{t^-}(e_{k}-e_{i}))-\phi({x}_{t^-})\right]\,dM^{ik}_t, \]
with $\displaystyle  M^{ik}_t= N^{ik}_t - \lambda_{ik}\int_0^t f_i( x_{s})\, ds.$
The S.D.E.~equation (\ref{quasispeciesstoch}), endowed with any random initial position $\mathcal{F}_0$-measurable, is well-posed in the weak sense stated by next lemma.
Moreover it is consistent with frequency modeling, i.e.~when the starting position is in $S$ then the solution stays in $S$ for all $t\ge 0$, \textit{a.s.}.
Actually, the next Lemma holds.
\begin{lemma}\label{mg-probl}
Let $f:\mathbb{R}^{\d+1}\rightarrow \mathbb{R}^{\d+1}$ be a Lipschitz continuous function with compact support containing the simplex $S$ and such that $|f(x)|\neq 0$ if $x\in S$. Let  $B$ be the operator defined by
\[B\phi(x)= - a(x) \cdot \nabla \phi(x) + \mathbb{I}_{S}({x})\mathcal{J}\phi(x). \]
Then for every probability measure $\pi_0$ assigned on $(\mathbb{R}^{\d+1},{\mathcal B}(\mathbb{R}^{\d+1}))$, the martingale problem for $(B,\pi_0)$ is well-posed, that is there exists
a filtered probability space and a $(\d+1)$-dimensional Markov process $(X_t,\;t\geq 0)$ with initial distribution $\pi_0$ on it such that
\[ \phi(X_t)-\phi(X_0)-\int_0^tB\phi(X_s)\,ds \] is a martingale.
This process is unique in law, that is if $(\tilde{X}_t,\;t\geq 0)$ is a different solution of the
martingale problem for $(B,\pi_0)$, then $X$ and $\tilde{X}$ have
the same finite dimensional distributions.\\
Moreover, if the support of $\pi_0$ is contained in $S$ then for
all $t\geq 0$, $X_t\in S$ \textit{a.s.}.
\end{lemma}
\begin{proof}
For all $\phi\in\mathcal{C}^1_b(\mathbb{R}^{\d+1})$,  $\mathbb{I}_{S}({x})\mathcal{J}\phi(x)$ can be written as
\begin{align*}
\lambda(x)
\int_{\mathbb{R}^{\d+1}}
\Big[\phi\big({x}+\mathbb{I}_{S}({x}){y}\big)-\phi({x})\Big]\,m_{x}(d{y})
\end{align*}
where
\[\lambda(x)=\sum_{i\neq k}\lambda_{ik}f_i(x)\]
and  $m_{x}(d{y})$ is the probability measure on
$\big(\mathbb{R}^{\d+1},\mathcal{B}(\mathbb{R}^{\d+1})\big)$ defined
 by
\begin{equation}\label{eq-misura}
m_{x}(d{y})=
\begin{cases}
&\displaystyle\sum_{i\neq k}\displaystyle\frac{\lambda
_{ik}f_i(x)}{\lambda(x)}\delta_{\gamma_{ik}x_i(e_{k}-e_{i})}(dy), \
\ \text{ if } x\in S\\
&\text{any probability measure, if }  \  x\not\in S.
\end{cases}
\end{equation}
Note that by the hypotheses $\lambda(x)\neq 0$ when $x\in S$.

Moreover it is easy to see by using Skorohod construction for random variables that  $\mathbb{I}_{S}({x})\mathcal{J}\phi(x)$ can also be expressed in the form
  \[\int_\Xi \big[\phi\big(x+K(x,\zeta)\big)-\phi(x)\big]\,\nu(d\zeta)\] where
$(\Xi,\Upsilon)$ is a measurable space, $(x,\zeta)\rightarrow
K(x,\zeta)$ is a measurable bounded function on
$\mathbb{R}^{\d+1}\times\Xi$ with values in $\mathbb{R}^{\d+1}$ and
$\nu(d\zeta)$ is a $\sigma$-finite measure on
$(\Xi,\Upsilon)$.

More precisely the previous equality holds for
$(\Xi,\Upsilon)=\big((0,1)^{\d}\times\mathbb{R}^+,\mathcal{B}((0,1)^{\d})\otimes\mathcal{B}(\mathbb{R}^+)\big)$, with general element denoted by $\zeta=(u_0,u_1,\ldots
u_{\d-1},\theta)$, 
\[ \nu(d\zeta)=du_0 du_1\ldots du_{\d-1}\, d\theta, \] and the
function $K$ constructed as follows (see \cite{CN}).
Fixed $x\in S$, let $Y=(Y_0,...,Y_{\d})$ be a random vector with
law $m_{x}(d{y})$ defined by (\ref{eq-misura}). Let $y_0\rightarrow F_0(y_0)$ be
the distribution function of  $Y_0$ and moreover, for $n=1,\ldots
, \d$, let $F_n(y_n\mid y_0,\ldots ,y_{n-1})$ be the distribution
function of  $Y_n$ given $Y_0=y_0,\ldots ,Y_{n-1}=y_{n-1}$, so
that $$m_x(dy)=F_0(dy_0)F_1(dy_1\mid y_0)\ldots
F_{\d}(dy_{\d}\mid y_0,\ldots ,y_{\d-1}).~$$
Finally denote by
$F_n^{-1}$ the generalized inverse of $F_n$, for $n=0,1,\ldots ,
\d$.
Then  for $x\in S$ let $K(x,\cdot):
\Xi\rightarrow\mathbb{R}^{\d+1}$ be
 defined by
\begin{align*}
&K_0({x},\zeta)=F_0^{-1}(u_0),\cr
&K_1({x},\zeta)=F_1^{-1}\big(u_1\mid K_0({x},\zeta)\big),\cr
 &K_n({x},\zeta)=F_n^{-1}\big(u_n\mid K_0({x},\zeta), K_1({x},\zeta),\ldots, K_{n-1}({x},\zeta)\big),\;\;\;n=1,\ldots,\d-1,
\end{align*}
and
$$K_{\d}({x},\zeta)=\mathbb{I}_{(0,\lambda(x))}(\theta) F_{\d}^{-1}\Big(
\theta / \lambda(x)\mid K_0({x},\zeta), K_1({x},\zeta),\ldots,
K_{\d-1}({x},\zeta)\Big).$$
For $x\notin S$ let  $K(x,\cdot)$ be identically zero.

It is to remark that the above construction implies that, for $x\in S$, we have
\begin{equation}\label{nu-and-lambda}\nu\big(\zeta\in\Xi,
|K(x,\zeta)|\neq 0\big)=\lambda(x).\end{equation}
In fact the distribution functions used in the construction have no jumps at zero so that none of the general inverses yields zero on a set of positive measure.  Moreover for all
$x$ and $\zeta$, $|K(x,\zeta)|\leq 2$.
Then existence of a solution of the martingale problem for $(B,\pi_0)$ follows by
existence of a solution of the SDE
\begin{equation}\label{vect-freq-eq-integral-2}
X_t=X_0 - \int_0^t a(X_s)\,ds+\int_0^t\int_\Xi
K(X_{s^-},\zeta)\,\mathcal{N}(ds\times d\zeta),\;\;t\geq
0\end{equation} where $\mathcal{N}(dt\times d\zeta)$ is a Poisson
random measure on
$\big(\mathbb{R}^{\d+1},\mathcal{B}(\mathbb{R}^{\d+1})\big)$ with
mean measure $dt\times \nu(d\zeta)$ and $X_0$ is a random variable
on the same probability space with distribution $\pi_0$.~A strong
non-explosive Markov solution of (\ref{vect-freq-eq-integral-2})
 exists by \cite{AKK}, where more
general SDE with jumps are treated.~Indeed, under our regularity assumption
on $f$, not only the drift coefficient verifies sub-linear growth
and Lipschitz condition but also the intensity of the
point process which counts the total number of jumps in (\ref{vect-freq-eq-integral-2}) is bounded.~In fact
 (\ref{nu-and-lambda}) joint with the regularity of $f$ gives
\[
\sup_{x\in S}\nu\big(\zeta\in\Xi, |K(x,\zeta)|\neq 0\big)=\sup_{x\in
S}\lambda(x)<+\infty.
\]
Again we refer to \cite{AKK} for deriving uniqueness in law of
$(X_t,\;t\geq 0)$.

Finally, following \cite{AKK}, the construction of the strong
solution of equation (\ref{vect-freq-eq-integral-2}) uses
sequentially the deterministic and the stochastic part of the
dynamic. So if the support of $\pi_0$ is contained in $S$ then
every trajectory of the solution verifies $X_t\in S$, for all
$t\geq 0$, when either the deterministic dynamic or the stochastic
dynamic do not allow to the trajectories starting in $S$ to leave
$S$. As it is well-known this is true for the deterministic
dynamic. As far as the stochastic dynamic is concerned, it is
sufficient to note that when $x\in S$ then
$x+\gamma_{ik}x_i(e_{k}-e_{i})\in S$.
\end{proof}
\begin{proposition}The frequencies process is
well-defined in law for all $t\geq 0$.
\end{proposition}\begin{proof}It follows immediately by previous
lemma recalling (\ref{mg-eq}), i.e.~that the frequencies process
starting at $x\in S$ solves for all $t\geq 0$ the martingale
problem for
 $(B,\delta_x)$.\end{proof}
\begin{remark}\label{confrontoquasispecie}
Let us remark that when the fitness vector $f$ assumes constant
value on $S$, then the point processes $N^{ik}_t$ are classical Poisson processes of parameter $\lambda_{ik}f^i$ and (\ref{quasispeciesstoch}) can be
written in form (\ref{vect-freq-eq-integral-2}) with
\[\Xi=Z=\{(i,k)\in \{0,\ldots,\d\}^2,\;i\neq k\},\;\;\;\;
K(x,(i,k))=\mathbb{I}_{S}(x)\gamma_{ik}x_i(e_{k}-e_{i})\]
\[\mathcal{N}([0,t]\times
(i,k))=N^{ik}_t,\;\;\;\;\nu(\{(i,k)\})=\lambda_{ik}f^i.\]
 So \cite{AKK} directly applies to our model, and the
frequencies process is the strong solution of equation
(\ref{quasispeciesstoch}) so that
\begin{align*}
dx^k_t = &   \left[ \left(f_k(x_t) - \bar{f}(x_t)\right) x^k_t +
\sum\limits_{i=0}^{\d} f_i(x_{t} ) \lambda_{ik} \gamma_{ik} x^i_t
\right] dt\\ \label{quasispeciesstochfalsa} & + \sum\limits_{i\neq k} \gamma_{ik} x^i_{t^-} dM^{ik}_t -\sum\limits_{i\neq k}
\gamma_{ki} x^k_{t^-} dM^{ki}_t.
\end{align*}
By taking  $m_{ik}= \lambda_{ik} \gamma_{ik}$  as $i \neq k$, $
m_{kk}= -\sum\limits_{i\neq k} \lambda_{ki} \gamma_{ki}$,  the
stochastic dynamics \eqref{quasispeciesstoch} is nothing but the
quasispecies equation \eqref{quasispecies1}, perturbed by a
martingale term.
\end{remark}

\section{The Kolmogorov equation}\label{sec3}

We next address to the  expected value of the frequencies process
\[ u_k(x,t) = \ecal\left( x^k_t \, \big| \, x_0= x \right) , \]
as $k=0,1,\dots \d$,
and deduce rigorously that it satisfies its Kolmogorov equation.
\begin{proposition}\label{Kolm}
For each $k=0,1\dots\d$ we have 
\[
\left\{\begin{array}{ll}
 \partial_tu_k(x,t) + {a}(x) \cdot \nabla u_k(x,t) ={\mathcal J} u_k (x,t), \qquad & x\in S, \; t>0 \\
u_k(x,0)=x_k, & x\in S ,\; t=0, \end{array}\right.
\]
 \end{proposition}
 We recall that $a$ and  $\jcal$ have been introduced in \eqref{a} and \eqref{J}, respectively.

Proposition \ref{Kolm} follows readily by next result about the semigroup on $B(\mathbb{R}^{D+1})$ defined by
\[ T_t\phi(x)=\mathcal{E}\left(\phi(X^x_t)\right).\]
Here, $X^x_t$ denotes the solution at time $t$ of \eqref{quasispeciesstoch}, with deterministic starting position $x\in \mathbb{R}^{\d+1}$.

\begin{lemma} Let $f$ and $B$  be as in Lemma \ref{mg-probl}.
Then $B$ is the infinitesimal generator of $(T_t,\;t\geq 0)$ with domain $\mathcal{D}(B)$ containing $\mathcal{C}^2_K(\mathbb{R}^{D+1})$. Moreover for each
$\phi\in \mathcal{C}^2_K(\mathbb{R}^{D+1})$ the scalar function $(x,t)\rightarrow
u_{\phi}(x,t)=T_t\phi(x)$ satisfies the Kolmogorov equation
\[\left\{\begin{array}{ll}
\partial_t  u_\phi (x,t) + a(x) \nabla u_{\phi}(x,t) = \mathcal{J}u_{\phi}(x,t),
\\
u_{\phi}(x,0)=\phi(x).
\end{array}\right.\]
\end{lemma}
\begin{proof}
Any function $\phi\in \mathcal{C}^2_K$ belongs to the domain of $B$ since
\[
\lim_{t\rightarrow 0^+}\sup_{x\in  \mathbb{R}^{D+1}}\Big|\frac{T_t\phi(x)-\phi(x)}t-B\phi(x)\Big|=0.
\]
In fact Ito's formula joint with Fubini's theorem  gives
\begin{equation}\label{semigr-eq}T_t\phi(x)=\phi(x)+\int_0^tT_s B\phi(x)\,ds\end{equation} so that the above limit coincides with
\begin{equation}\label{lim-generatore}
\lim_{t\rightarrow 0^+}t^{-1}\sup_{x\in \mathbb{R}^{D+1}}\Big|\int_0^t \mathcal{E}\left(B\phi(X^x_s)-B\phi(x)\right)ds\Big|=0.
\end{equation}
Then equality (\ref{lim-generatore}) follows by considering that for all $x\in \mathbb{R}^{\d+1}$ the
regularity of $f$ and $\phi$ implies that
 $$\mathcal{E}\big(|B\phi(X^x_s)-B\phi(x)|\big)\leq C(\phi)\,s,$$
with $C(\phi)$ a positive constant depending on $\phi$.
Finally the thesis follows from (\ref{semigr-eq}) by recalling that
(see, e.g. \cite{Lamp}) if $\phi\in\mathcal{D}(B)$ then
$T_t\phi\in\mathcal{D}(B)$ and $BT_t\phi=T_tB\phi$, for each $t\geq
0$.
\end{proof}

\subsection{Dimensional reduction}
We remark that the number of variable can be reduced by setting $x_0=1-\sum\limits_{k=1}^{\d}x_k$. The new variable, that we still denote by $x=(x_1\dots x_{\d})$, lives in the closed set 
\[ \Sigma=\{x=(x_1\dots x_{\d}) \in \R^{\d} \, : \, x_k \ge 0 \, \mbox{ as } k=1\dots\d, \,  \sum\limits_{k=1}^{\d}x_k \le 1\}.\]
In all the following, we continue to write ${a}$ and $f$ for the respective functions depending on the $\d$ variables  $x=(x_1\dots x_{\d})\in \Sigma$.
As or the non-local term $\jcal$, it becomes
\begin{align*}
 {\mathcal J} \phi(x) = & \sum\limits_{\substack{i,j=1\dots\d\\i\neq j}}\lambda_{ij} f_i(x) \left[ \phi\left(x+ \gamma_{ij}x_i ({\bf e_j}-{\bf e_i})\right)- \phi(x) \right] \\
& +\sum\limits_{i=1\dots\d}\lambda_{i0} f_i(x) \left[ \phi\left(x- \gamma_{i0}x_i {\bf e_j}\right)- \phi(x) \right] \\
& + \sum\limits_{j=1\dots\d}\lambda_{0j} f_0(x) \left[ \phi\left(x+ \gamma_{0j}(1-\sum\limits_{i=1\dots\d}x_i) {\bf e_j}\right)- \phi(x) \right]
 ,
\end{align*}
for any continuous scalar function $\phi\in{\mathcal C}(\Sigma;\R)$.

Next subsections are devoted to the analytical study of the decoupled system
\begin{align} \label{K}
\left\{\begin{array}{ll}
 \partial_tu_k(x,t) + {a}(x) \cdot \nabla u_k(x,t) ={\mathcal J} u_k (x,t), \qquad & x\in \Sigma, \; t>0 \\
u_k(x,0)=x_k, & x\in \Sigma,\; t=0, \end{array}\right.
\end{align}
as $k=1\dots\d$.

\subsection{Mild solutions}

We establish global well-posedness of  the Cauchy problem \eqref{K}.
As a preliminary, we notice that the vector field ${a}(x)$  does not point outward at the boundary of $\Sigma$. 

\begin{remark}\label{edges}
The projection of the vector field $a$ orthogonal to the sides of the  boundary of $\Sigma$ is always zero.
Actually at $ x_k=0$ we have ${a}(x) \cdot {\bf e_k} =   a_k(x)=0$, while at $\sum\limits_{k=1}^{\d}x_k=1$ we have 
\[ {a}(x) \cdot \left(\sum\limits_{k=1}^{\d}{\bf e_k}\right) = \sum\limits_{k=1}^{\d} a_k (x)
=  -\sum\limits_{k=1}^{\d} x_k \, f_k(x) + \bar{f}(x)  \sum\limits_{k=1}^{\d} x_k   =  - x f(x) + \bar{f}(x) =0,
\]
because $x_0=1- \sum\limits_{k=1}^{\d}x_k =0$.
\end{remark}

\remove{DOPO \begin{definition}\label{ASiE}
We say that a strategy $x^{\ast}\in T$ is {\bf{asymptotically stable in expectation}} ({\bf{ASiE}}) if there is $\ep>0$ so that $u(x,t)\to x^{\ast}$ as $t\to+\infty$ for all $x\in T$ with $|x-x^{\ast}|<\ep$. 
\end{definition}}
As a consequence,  the flux of the Cauchy problem for the autonomous equation
\begin{equation}\label{char} \left\{\begin{array}{l}
 \dot y   = {a}(y) , \\
 y(0)=x , 
\end{array}\right.\end{equation}
is well defined and maps $\Sigma$ into itself. It is worst noticing that it is nothing than the solution to the replicator equation \eqref{replicator}, after the dimensional reduction.
In the following, we shall write $Y(x,t)$ for the solution of \eqref{char} starting at $x$. 
 For any $x,t$, the function $s\mapsto Y(x,s-t)$ is the characteristic line of through $x,t$ for problem \eqref{K}.

\begin{definition}
A {\bf{mild solution}} is a function $u\in \ccal\left([0,T)\times\Sigma ; \R^{\d}\right)$ satisfying the integral formula
\begin{align}\label{uformula} 
& u_k(x,t)= Y_k(x,-t) + \int_0^t \jcal u_k\left(Y(x,s-t),s\right) ds ,
\end{align}
as $k=1\dots\d$.
\end{definition}

Following the method of characteristics, we define \[ v(x,s,t)= u(Y(x,s-t),s)\]  and notice that for any $t$, $v$ solves an integro-differential problem
\begin{align}\label{ipde}
 \left\{\begin{array}{ll}
\partial_s v_k(x,s,t)  
 = \Ical v_k(x,s,t)  \qquad \quad & x\in\Sigma, \; s >0 ,\\
v_k(x,0,t)= Y_k(x,-t) &  x\in \Sigma , \; s=0 ,\end{array}\right.
\end{align}
as $k=1\dots\d$. Here 
\begin{align*}
 \Ical v_k(x,s,t)  = & \sum\limits_{\substack{i,j=1\dots\d\\i\neq j}}\lambda_{ij} f_i(Y(x,s-t)) \left[ v_k(x+\delta_{ij}(x,s-t),s,t)-v_k(x,s,t)\right] \\
& +\sum\limits_{i=1\dots\d}\lambda_{i0} f_i(Y(x,s-t)) \left[ v_k(x+\delta_{i0}(x,s-t),s,t)-v_k(x,s,t)\right] \\
& + \sum\limits_{j=1\dots\d}\lambda_{0j} f_j(Y(x,s-t)) \left[ v_k(x+\delta_{0j}(x,s-t),s,t)-v_k(x,s,t)\right]
 , \\
\delta_{ij}(x,r) = & Y\left(Y(x,r)+\gamma_{ij}Y_i(x,r)({\bf e_j}-{\bf e_i})\,  , \, -r\right) -x, \\
\delta_{i0}(x,r) = & Y\left(Y(x,r)-\gamma_{i0}Y_i(x,r){\bf e_i}) \, , \, -r\right) -x, \\
\delta_{j0}(x,r) = & Y\big(Y(x,r)+\gamma_{0j}(1-\sum\limits_{i=1\dots\d}Y_i(x,r)){\bf e_j} \, , \, -r\big) -x.
\end{align*}

The Cauchy problem \eqref{ipde}  is seen in the mild sense, actually
\begin{align} \label{mildipde} 
 v_k(x,s,t)= & Y_k(x,-t) + \int_0^s \Ical v_k(x,r,t) dr . 
\end{align}
We next use a fixed point argument to solve \eqref{ipde}. This brings a mild solution also to \eqref{K} because
\[ v_k(x,t,t)= u_k(x,t)  .\]
The following proof is standard, but we report it for the sake of completeness.

\begin{proposition}\label{wpmild}
The problem \eqref{K} has an unique global solution $u\in \ccal\left(\Sigma\times[0,+\infty);\R^{\d}\right)$.
\end{proposition}
\begin{proof}
Let $S,T>0$, and $\chi$ the set of continuous scalar functions of $(x,s,t)\in\Sigma\times[0,S]\times[0,T]$,
which is a Banach space endowed with the sup-norm.
For $k=1\dots\d$ and $r>0$ , let $B^k_r$ be the set  of functions of $\chi$ such that 
\[ \sup\{ \|v(x,s,t)-Y_k(x,-t) \| \, : \, x\in\Sigma, 0\le s\le S , 0\le t\le T \} \le r ,\]
and ${\mathcal T}$ the operator
\[{\mathcal T} v (x,s,t)= Y(x,-t)+ \int_0^s \Ical v(x,r,t) \, dr . \]
${\mathcal T}$ maps $B_r$ into itself because for $v\in B_r$ we have 
\begin{align*}
\|{\mathcal T} v - v_o\| \le 2 \|f\|_{\infty} s \| v\| \le 2\|f\|_{\infty} S (r+\|v_o\|) \le r ,
\end{align*}
provided that $S\le r/2\|f\|_{\infty}(r+\|v_o\|)$.
Moreover ${\mathcal T}$ is a contraction since
\begin{align*}
\|{\mathcal T} v - {\mathcal T} w \| \le 2 \|f\|_{\infty} s \|v-w\| \le 2\|f\|_{\infty} S \|v-w\| \le \frac{r}{r+\|v_o\|} \|v-w\| .
\end{align*}

It follows by contraction Theorem that \eqref{ipde} has an unique mild solution $v\in \ccal\left([0,S]\times[0,S] \times\Sigma \right)$ at least for $S=r/2\|f\|_{\infty} (r+\|v_o\|)$.
Moreover, since $v\in B_r$, we have $\|v\|\le r+\|v_o\|$. Then the fixed point argument can be iterated to get, at any step $n$, a solution defined until $S_{n} = \frac{1}{2\|f\|_{\infty}}\left(\frac{r}{r+\|v_o\|}+ \sum\limits_{k=1}^{n}\frac{kr+\|v_o\|}{(k+1)r+\|v_o\|}\right)$. As $S_n\to +\infty$, existence and uniqueness of a global solution follows.
\end{proof}

\subsection{Viscosity solution}
Using the tools of viscosity theory, one can eventually prove the following result

\begin{proposition}\label{comparison}
 Let $u$ and $v$ be, respectively, viscosity sub/supersolutions of  \eqref{K} with $u(x,0)\le v(x,0)$ for all $x\in \Sigma$. Then $u(x,t)\le v(x,t)$ for all $x\in \Sigma$ and $t>0$.
\end{proposition}
\begin{proof}
The proof consists in assuming that there exists $T>0$ so that 
\[ M = \max \{ e^{-t}\left(u (x,t)- v(x,t)\right)  \, : \, 0\le x\le 1 , \, 0\le t\le T \}>0 \]
and getting a contradiction.
Standard arguments exclude that any maximum point $\bar{x},\bar{t}$ may have $\bar t = 0$ either $\bar t \in(0,T]$ and $\bar x$ in the interior of $\Sigma$. 
We show that neither  $\bar x$ in the boundary of $\Sigma$ is allowed. 
To fix idea, suppose that $\bar x$ is a strict maximum point with $\bar x_1=0$, $\bar x_i>0$ as $i=2,\dots \d$  and $\sum\limits_{i=1}^{\d}\bar x_i=1$. Next, a barrier function
\[\kappa(x,t) = 1/{x_1}  +1/({1-\sum\limits_{i=1}^{\d}x_i})+1/(T-t) \] 
is introduced and $M$ is approximated by $M(\delta)$, the maximum value of
 \[ e^{-t}\left(u (x,t)- v(x,t)\right) - \delta \kappa(x,t) .\]
It is clear that there exists a maximum point $x(\delta), t(\delta)$ with $x(\delta)$ in the interior of $\Sigma$ and $0\le t(\delta)<T$. Moreover 
\[
M(\delta)\to M>0, \quad x(\delta)\to \bar x , \quad t(\delta)\to \bar{t}>0, \quad \delta\,\kappa(x(\delta),t(\delta)) \to 0\]
as  $\delta \to 0$.
Now the  perturbed functions $u_{\delta}(x,t)=  e^{-t} u (x,t) - \delta\,\kappa(x,t)$ and $v_{\delta}(x,t)=e^{-t}v(x,t)$ can be handled by the standard tool of doubling variables and using equation \eqref{K}. The step of passing to the limit as $\delta\to 0$ can be performed since
\begin{align*}
\left|\delta a(x(\delta)) \cdot D_x\kappa(x(\delta),t(\delta))\right|\le  & \left(\left|\dfrac{a_1(x(\delta))}{x_1(\delta)} \right|+ \left|\dfrac{\sum\limits_{i=1}^{\d}a_i(x(\delta))}{\sum\limits_{i=1}^{\d}(\bar x_i-x_i(\delta))}\right|\right) \delta \, \kappa(x(\delta),t(\delta)),
\end{align*}
where the quantities  $|{a_1(x(\delta))}/{x_1(\delta)}|$ and $|{\sum\limits_{i=1}^{\d}a_i(x(\delta))}/{\sum\limits_{i=1}^{\d}(\bar x_i-x_i(\delta))}|$ are bounded because $a$ is Lipschitz continuous and $a_1(\bar x)$, $\sum\limits_{i=1}^{\d}a_i(\bar x)$ are zero by Remark \ref{edges}.
\end{proof}

In our particular setting, we have decided to emphasize the transport component by following the method of characteristics. Nevertheless, the solution defined and produced in the previous section coincides with the viscosity solution.
\begin{proposition} Mild and viscosity solution of \eqref{K} coincide.
\end{proposition}
\begin{proof}
As well posedness holds in both framework, it suffices to check that the mild solution  is a solution in viscosity sense. We only prove the subsolution part, because the supersolution is identical.
Let $k=1\dots\d$, and $\phi$ be a smooth scalar function such that $u_k-\phi$ has a global maximum at  $(x,t)\in \Sigma\times(0,+\infty)$, with $u_k(x,t)=\phi(x,t)$.
First, we notice that $\jcal u_k(x,t) \le \jcal\phi(x,t)$.
Next, we set $\psi(x,s,t)= \phi(Y(x,s-t) ,s)$: it is clear that also $v_k-\psi$ has a global maximum point at $(x,t,t)$, and therefore $\Ical v_k(x,t,t) \le \Ical \psi(x,t,t)$. Hence it is easily seen that
\begin{align*}
\partial_t\phi(x,t) + {a}(x) \cdot D_x \phi(x,t)=\partial_s \psi({x},{t},{t}) = \partial_s v_k(x,t,t) \\ = \Ical v_k(x,t,t)= \jcal u_k(x,t) \le \jcal \phi(x,t).
\end{align*}
\end{proof}
Comparison techniques provides further  information about the regularity of $u$ w.r.t.~$ x$ (see, for instance, \cite{Ama03}). 
\begin{proposition}\label{lipschitz}
The solution to \eqref{K} is Lipschitz continuous w.r.t.~$x$, for every fixed $t>0$.
\end{proposition}

\begin{remark}\label{generalcase}
It is not hard to extend all results in this section to a more general class of problems that can be written as \eqref{K}, 
where ${\mathcal J}$ is a first order non local operator in the form
\[ {\mathcal J} \phi (x) = \int f(x,\zeta) \left[ \phi\left(x+ {\bf\gamma} (x,\zeta)\right)- \phi(x) \right] d\nu(\zeta),\]
provided that
\begin{enumerate}[i)]
\item $\Sigma$ is a  compact set, whose boundary is globally Lipschitz, and consists in the union of smooth surfaces which have exterior normal vector.
\item ${a}$ is a Lipschitz continuous vector field defined on $\Sigma$. At every $x$ in the boundary of $\Sigma$, and for every exterior normal vector $n$, we have ${a}\cdot {\bf n} =0$.
\item $\nu$ is a finite measure.
\item ${\bf \gamma}$ is a (vector valued) function, Lipschitz continuous w.r.t.~$x\in \Sigma$ (uniformly w.r.t.~$\zeta$) with $x+{\bf \gamma}(x,\zeta) \in \Sigma$ for all $x\in \Sigma$ and $\nu$-almost any $\zeta$.
\item $f$ is a (real valued) function, Lipschitz continuous w.r.t.~$x\in \Sigma$ (uniformly w.r.t.~$\zeta$) with $f(x,\zeta) \ge 0$ for all $x\in \Sigma$ and $\nu$-almost any $\zeta$.
\item $u_o$ is a continuous (vector valued) function.
\end{enumerate}
\end{remark}

\section{Two quasispecies}\label{sec4}

With the aim of expounding the behavior of the expected frequencies, we deal with the simplest case: two species (i.e. $\d=1$) and constant fitness (i.e. $f_{i}(x)=f_i$, as $i=0,1$).  We also introduce the selection rate 
\[ \s=f_0-f_1 .\]
 To fix ideas we take $\s>0$, so that  $x_0=1$, $x_1=0$ is the only asymptotically stable rest point for the replicator equation \eqref{replicator}.

After the dimensional reduction $x=x_1$,  $x_0=1-x$, the (scalar) Kolmogorov equation \eqref{K} for \[ u(x,t)=\ecal\left(x(t)\, \big| \, x(0)=x\right)\] reads
\begin{align}\label{K1} 
\left\{\begin{array}{ll}
 \partial_tu +  \s (1-x) x \,  \partial_x u = \lambda_0f_0 \jcal_0u +\lambda_1f_1 \jcal_1u , & 0\le x\le 1 , \, t>0, \\[.15cm]
u(x,0)=x & 0\le x\le 1 .,\end{array}\right.
\end{align}
where
\begin{align*}
 \jcal_0 u(x,t)= u(x + \gamma_0 (1-x),t)-u(x,t), \quad
 \jcal_1 u(x,t)=u(x-\gamma_{1} x,t)-u(x,t).
\end{align*}
Here and henceforth we have shortened notations by writing $\gamma_0=\gamma_{01}$, $\gamma_1=\gamma_{10}$,  $\lambda_0=\lambda_{01}$, $\lambda_1=\lambda_{10}$.

Our main concern is to compare the dynamics  with rare mutation and the standard  quasispecies dynamics.
In this particular setting (and after the dimensional reduction) the quasispecies equation \eqref{quasispecies1} reads
\begin{equation}\label{quasispecies2}\left\{\begin{array}{ll}
\dot x = - \s x(1-x) + m_0 f_0(1-x) - m_1 f_1 x \quad  & t>0, \\
x(0)=x.& \end{array}\right.\end{equation}
Here $m_0$  stands for the mutation parameter from species 0 to species 1, and viceversa for $m_1$. With Remark \ref{confrontoquasispecie} in mind, we take
\[ m_i=\lambda_i\gamma_i, \; \text{ as } i=0,1 .\]
Let $X(x,t)$ stand for the flux associated to \eqref{quasispecies2}, i.e.~the solution to the o.d.e.~with initial condition $x$,  computed at $t$.
It solves the homogeneous transport equation
\begin{align}\label{qs1} 
\left\{\begin{array}{ll}
 \partial_tX +  \left(\s x(1\!-\!x) - m_0 f_0(1\!-\!x) + m_1 f_1 x\right) \partial_x X = 0, \quad & 0\le x\le 1 , \, t>0, \\[.15cm]
X(x,0)=x , & 0\le x\le 1 .\end{array}\right.
\end{align}

Let us recollect some basic facts about \eqref{K1}.
In this simple scalar setting,  $u$ is a classical smooth solution.
\begin{proposition}\label{classicalsolution}
The Cauchy problem \eqref{K1} admits a classical solution $u\in \ccal^{\infty}([0,1]\times[0,\infty))$.
\end{proposition}
We do not report in details the proof of this result because it is completely standard: it relies in deriving w.r.t.~$x$ iteratively the equation and the initial datum in \eqref{K1} and noticing that the obtained problem inherits the same structure and regularity.
We rather go into details and obtain some more estimates  concerning first and second order derivative w.r.t.~$x$.

\begin{lemma}\label{u_xest}
For every $t>0$, the solution to \eqref{K1} satisfies \[ 0\le\partial_xu(x,t)\le e^{(\s-m_0f_0-m_1f_1)t}\] for all $x\in[0,1]$.
\end{lemma}
\begin{proof}
Let $p=\partial_xu$: deriving \eqref{K1} w.r.t.~$x$ gives
\begin{align}\label{K1_x} 
\left\{\begin{array}{l}
 \begin{array}{r}\partial_tp +  \s (1-x) x \,  \partial_x p +\left(\s(1-2x) + m_0f_0 +m_1f_1\right)p \\[.15cm]
= \lambda_0f_0(1-\gamma_0) \jcal_0p +\lambda_1f_1 (1-\gamma_1)\jcal_1p , \end{array}\\[.15cm]
p(x,0)=1 .\end{array}\right.
\end{align}
It is clear that $p=0$ and $p=e^{(\s-m_0f_0-m_1f_1)t}$ are, respectively, a subsolution and a supersolution. So the thesis follows by comparison.
\end{proof}

A similar estimate holds also in the general case treated in previous section, and can be proved for viscosity solutions. 
It is worst mentioning that, in particular, $u$ is monotone increasing w.r.t.~$x$, for every fixed $t>0$.
It is also convex, as shown by next lemma.

\begin{lemma}\label{convex}
For every $t>0$, the solution to \eqref{K1} is convex w.r.t.~$x$. Moreover there exist two constants $\c>0$ and $\mu\in\R$ such that 
\[ 0\le\partial^2_{xx}u(x,t)\le \c e^{\mu t}\] for all $x\in[0,1]$.
\end{lemma}
\begin{proof}
Deriving \eqref{K1} twice w.r.t.~$x$ gives that $q=\partial^2_{xx}u$ solves 
\begin{align*}
 \partial_tq +  \s x (1-x) \partial_x q +(4\s(1-x)+\alpha) q =  2\s \partial_x u + \\[.15cm]
 \lambda_0f_0(1-\gamma_0)^2 \jcal_0q +\lambda_1f_1 (1-\gamma_1)^2\jcal_1q  ,
\end{align*}
with $\alpha=-2\s+(2-\gamma_0)m_0f_0+(2-\gamma_1)m_1f_1$. As $\partial_xu\ge 0$, the function $\partial^2_{xx}u$ is a supersolution to the homogeneous Cauchy problem
\begin{align*}
\left\{\begin{array}{l}
 \partial_tq +  \s x (1-x) \partial_x q +  (4\s(1-x)+\alpha) q = \lambda_0f_0(1-\gamma_0)^2 \jcal_0q +\lambda_1f_1 (1-\gamma_1)^2\jcal_1q  , \\[.15cm]
q(x,0)=0 ,\end{array}\right.
\end{align*}
and therefore $\partial^2_{xx}u\ge 0$.
On the other hand, as $\partial_xu \le e^{(\s-m_0f_0-m_1f_1) t}$,  the function $\partial^2_{xx}u$ is a subsolution to
\begin{align}\label{inhp}
\left\{\begin{array}{l}
 \begin{array}{r} \partial_tq +  \s x (1-x) \partial_x q +(4\s(1-x)+\alpha) q = 2\s e^{(\s-m_0f_0-m_1f_1) t} +\\[.15cm]
\lambda_0f_0(1-\gamma_0)^2 \jcal_0q +\lambda_1f_1 (1-\gamma_1)^2\jcal_1q  ,
\end{array}\\[.15cm]
q(x,0)=0 .\end{array}\right.\end{align}
Eventually, also  the estimate from above of $\partial ^2_{xx}u$ follows by comparison, after having checked that the function $\overline q (x,t)= \dfrac{2\s}{\ep} e^{(2\s-m_0f_0-m_1f_1+\ep) t}$ is a supersolution to \eqref{inhp}, for every $\ep>0$,
\end{proof}

It follows that the expected value of the density with rare mutations  is greater or equal than the deterministic one, i.e.~rare mutations increase the survival opportunities of the low-fitness species.

\begin{proposition}\label{below}
Let $u$ and $X$ be, respectively, the solution to \eqref{K1} and \eqref{qs1}. Then $1\ge u(x,t)\ge X(x,t)$  for all $0\le x\le 1$ and $t\ge 0$.
\end{proposition}
\begin{proof}
It is clear that $u(x,t)\le 1$, because the constant function $1$ is a supersolution to \eqref{K1}.
Concerning the estimate from below, it follows  by comparison after checking that $u$ is a supersolution of  \eqref{qs1}. Indeed
\begin{align*} 
 \partial_tu+ \left(\s x(1-x) - m_0f_0(1-x) +m_1 f_1 x\right)  \partial_x u = \\
\lambda_0f_0\left[ u(x+\gamma_0(1-x),t)-u(x,t)-\gamma_0(1-x)\partial_xu(x,t)\right]\\
+\lambda_1f_1\left[ u(x-\gamma_1x,t)-u(x,t)+\gamma_1x\partial_xu(x,t)\right] = \\
\lambda_0f_0 \int_0^1\partial_{xx}u(x+\theta\gamma_0(1-x),t) d\theta +
\lambda_1f_1 \int_0^1\partial_{xx}u(x-\theta\gamma_1x,t) d\theta  \ge 0
\end{align*}
 by convexity.
\end{proof}

\subsection{Large time behavior}
It is well known that the quasispecies equation \eqref{quasispecies2} has an asymptotic equilibrium  at the point $\bar x\in[0,1]$ singled out by the relation 
\[\s \bar x(1-\bar x) - m_0 f_0(1-\bar x) + m_1 f_1 \bar x=0 ,\]
and that its basin of attraction is given by $[0,1]$ or $[0,1)$,  depending on the value of the parameters (see Remark \ref{basin} later on). 
In our notation, this means that the  solution to \eqref{qs1} satisfies $\lim\limits_{t\to+\infty}X(x,t)=\bar x$ for every $x\in[0,1]$ (or for every $x\in[0,1)$).
It is interesting to study whether the rare mutation equation \eqref{K1} has the same asymptotic behavior, or rather exhibits a new equilibrium.
To begin with, we need to establish that the solution to \eqref{K1} actually admits a limit as $t\to+\infty$.
 For a restricted range of parameters, the estimate given in Lemma \ref{u_xest} suffices to deduce the large time behavior of $u$.
Otherwise, some more work is needed.
 
\begin{lemma}\label{u_xestfine}
If $\s\ge m_0f_0+m_1f_1$, then the solution to \eqref{K1} satisfies \[ \partial_xu(x,t)\le \frac{2}{1-x +e^{-(\s-m_0f_0) t}} e^{-m_1 f_1 t}\] 
for all $x\in[0,1]$ and $t\ge 0$. 
\end{lemma}
\begin{proof}
It suffices to check that $\overline p = 2e^{(\s-m_0f_0+m_1f_1 )t} /(1+ (1-x) e^{(\s-m_0f_0) t}) $ is a supersolution to \eqref{K1_x}.
To shorten notation, we write $z=(1-x) e^{(\s-m_0f_0) t}$. 
Trivially $\overline p(x,0) = 2/(2-x)\ge 1$. Moreover easy computations give that 
\begin{align*}
&\partial_t\bar p +  \s (1-x) x \,  \partial_x \bar p + (\s(1-2x)+m_0f_0+m_1f_1) \bar p  \\
& \qquad =  \dfrac{2 e^{(\s-m_0f_0+m_1f_1 )t}}{(1+z)^2}\left(2\s(1-x) +( (\s(1-x)+m_0f_0) z\right)\\
& \qquad \ge \dfrac{2e^{(\s-m_0f_0+m_1f_1 )t}}{(1+z)^2}m_0f_0 z , \\
&\jcal_1\bar p \le 0 , \\
&\jcal_0\bar p = \dfrac{2e^{(\s-m_0f_0+m_1f_1 )t}}{(1+z)(1+(1-\gamma_0)z)} \gamma_0 z  \le \dfrac{2e^{(\s-m_0f_0+m_1f_1 )t}}{(1+z)^2} \dfrac{\gamma_0 z}{1-\gamma_0},
\end{align*}
and the thesis follows immediately.
\end{proof}

We are now in the position to draw the large time behavior of $u$.

\begin{proposition}\label{ltb}
For every choice of the parameters, the function \[\bar u(x)=\lim\limits_{t\to+\infty}u(x,t)\] is well defined for every $x\in[0,1]$. 
If, in addition, $m_1>0$ or $m_1=0$ and $m_0\ge \s/f_0$, then the limit $\bar u$ is constant.
\end{proposition}
\begin{proof}
If $m_0f_0+m_1f_1>\s$, it follows by Lemma \ref{u_xest} via standard arguments that $u$ converges to a constant as $t\to+\infty$ (uniformly w.r.t.~$x$).

Similar statement follows by Lemma \ref{u_xestfine} if $m_0f_0+m_1f_1\le \s$ and $m_1>0$, provided that $x$ stay in any closed subset of $[0,1)$. Concerning the behavior at $x=1$, we deduce by equation \eqref{K1} that
\begin{align*}
\partial_t u(1,t)=\lambda_1 f_1 \left( u(1-\gamma_1,t) - u(1,t)\right) \le 0,
\end{align*}
as $u$ is increasing w.r.t.~$x$. Hence  $\bar u(1)=\lim\limits_{t\to+\infty}u(1,t)$ exists and is finite, actually
\begin{align*}
\bar u(1)=1-\lambda_1 f_1 \int_0^{+\infty}\left[ u(1,t)-u(1-\gamma_1,t) \right] dt.
\end{align*}
In particular
\begin{align*}
\lim\limits_{t\to+\infty}\left(u(1,t)-u(1-\gamma_1,t) \right)  = \bar u (1)-\bar u(1-\gamma_1) =0,
\end{align*}
 because the  function $t\mapsto u(1,t)-u(1-\gamma_1,t) $ has limit as $t\to+\infty$ and has finite integral on $[0,+\infty)$. 
This, in turns, implies that $\bar u$ is constant up to $x=1$. 

For the case $m_0=\s/f_0$ and $m_1=0$, we know by Proposition \ref{below} that $1\ge u(x,t)\ge X(x,t)$ for every $x\in[0,1]$ and $t\ge 0$.
But in this special case the asymptotically stable point of \eqref{qs1} is $\bar x =1$, hence 
\[ 1\ge u(x,t) \ge X(x,t)= 1-\dfrac{1-x}{1+(1-x)\s t} ,\]
 and $\bar u\equiv 1$.

The proof is completed by checking that $\lim\limits_{t\to+\infty}u(x,t)$ exists even if $m_0<\s/f_0$  and $m_1=0$.
In that case, it follows by Lemma \ref{u_xestfine} that $u(x,t)$ is equicontinuous w.r.t.~$x$ in any closed subset of $[0,1)$. Thus standard machinery for evolution equations yields that $u$  is equicontinuous w.r.t.~both $x$ and $t$ and therefore $\bar u(x)$ is well defined (and continuous) for $x\in[0,1)$.
On the other hand equation \eqref{K1} states that $\partial_t u(1,t) = 0$, so that $\bar u(1)=1=u(1,t)$ for all $t$.
\end{proof}

\begin{remark}\label{basin}
We mention in passing that Lemmas \ref{u_xest} and \ref{u_xestfine} imply that $u(x,t)\to \bar u$ uniformly w.r.t.~$x\in[0,1]$ if $\s\le m_0f_0+m_1f_1$, or, respectively, uniformly w.r.t.~$x$ in any closed set contained in $[0,1)$, if $\s>m_0f_0+m_1f_1$. This  behavior reflects that one of $X(x,t)\to\bar x$.
It is worth noting that, in case $m_0< \s/f_0$ and $m_1=0$, the quasispecies equation \eqref{qs1} gives
\[
X(x,t)=\bar x + \dfrac{x-\bar x}{1+\frac{1-x}{1-\bar x}\left(e^{\s(1-\bar x) t}-1\right)} ,
\]
with $\bar x=m_0f_0/\s<1$. Therefore the basin of attraction of $\bar x$ is only the interval $[0,1)$ and even the asymptotic limit of $X$  jumps from $\bar x$ to $1$ at $x=1$.
\end{remark}

For some choice of parameters, rare mutations give the same equilibrium of continuous mutations.
This happens, for instance, if $m_0=0$. In this case  the mutated descendants have higher fitness than their progenitors, and  mutation helps selection in fixing the high-fitness specie.

\begin{proposition}[Fair mutation]\label{fairmutation}
Assume that $m_0=0$, so that  the equilibrium for both the quasispecies dynamics \eqref{qs1} and the replicator equation \eqref{replicator} is $\bar x = 0$. The same holds also for \eqref{K1}, i.e.~$\bar u=0$.
To be specific, we have
\[
\dfrac{x}{1-x(1-e^{-\s t})} e^{-\s t} \ge u(x,t)\ge \dfrac{x}{1 - \dfrac{\s \, x}{\s+m_1f_1} (1-e^{-(\s+m_1f_1)})} e^{-(\s+m_1f_1)t}
\]
for all $t\ge0$.
\end{proposition}
It has to be remarked that the first and last  terms of the inequality  are the  solution to the replicator equation \eqref{replicator}, and  quasispecies equation \eqref{qs1}, respectively. 
\begin{proof}
As $m_0=0$, the Kolmogorov equation \eqref{K1} becomes
\begin{align*}
 \partial_tu + \s (1-x) x \,  \partial_x u = \lambda_{1}f_1 \left[  u((1-\gamma_{1} )x,t) -u(x,t)\right] \le 0
\end{align*}
because $u$ is increasing w.r.t.~$x$. Then  $u$ is a subsolution of the transport equation $\partial_tu + \s (1-x) x \,  \partial_x u =0$ and the first inequality follows.
In particular, we have that $\lim\limits_{t\to+\infty} u(x,t)=0$ for all $0\le x< 1$.
As for  $x=1$,  we know that
\begin{align*} 
 u(1,t) = 1-\lambda_{1}f_1  \int_0^t \left(  u(1,s)-u(1-\gamma_{1 },s) \right) ds.
\end{align*}
Since the integrand is nonnegative by monotonicity, the function $t\mapsto u(1,t)$ is monotone decreasing and bounded, so it converges.
Hence the function $t\mapsto  u(1,t)-u(1-\gamma_{1},t)$ is nonnegative, has finite integral in $[0,+\infty)$ and has limit as $t\to+\infty$. Eventually $\lim\limits_{t\to+\infty} u(1,t)=\lim\limits_{t\to+\infty} u(1-\gamma_{1},t)=0$.
The proof is now complete because the second inequality  has been established in Proposition \ref{below}.
\end{proof}

The large time behavior of rare mutations reflects the one of continuous mutation also in the opposite situation, i.e.~when mutation towards the low-fitness specie is so relevant to  overwhelm selection.

\begin{proposition}[Unfair mutation, strong case]\label{ltbsum}
Assume that $m_1=0$ and $m_0\ge \s/f_0$, so that the equilibrium of the quasispecies dynamics \eqref{qs1} is $\bar x = 1$. The same holds also for \eqref{K1}, namely $\bar u=1$. Moreover
\begin{align*}
 1\ge u(x,t) & \ge 1-  \frac{1-x}{1+\frac{\s(1-x)}{m_0f_0-\s}( e^{(m_0f_0-\s) t}-1)} \quad && \text{if } m_0>\s/f_0,
\intertext{or}
1\ge u(x,t) & \ge 1-  \frac{1-x}{1+\s(1-x)t} && \text{if } m_0=\s/f_0.
\end{align*}
\end{proposition}
Notice that the quantity in the right-hand side of both inequalities is the solution of the respective quasispecies equation.
\begin{proof}
The thesis follows by Proposition \ref{below}, because in this particular setting the  equilibrium condition for the standard quasispecies equation reads $(\s  x - m_0f_0)(1- x)$, and the only root contained in the segment line $[0,1]$ is $\bar x=1$.
\end{proof}

Something new happens when mutation is unfair (i.e.~$m_1=0$) but too weak to overwhelm selection (i.e.~$0<m_0<\s/f_0$).
 In this case, the behavior at large time  depends on the time intensity of the point process governing mutations, and it does not follow the relative quasispecies equation anymore.
As expected, the quasispecies equation is recovered as the time intensity goes to infinity.
This topic is  illustrated in next subsection. 

\subsection{Unfair, but weak, mutation}
We go into more details and inspect the case  $m_1=0$, $\s>m_0f_0$.
As only the coefficients $f_0$, $m_0$, $\gamma_0$ and $\lambda_0$ have effects, we shall omit to write the index ``$0$''.
The quasispecies equation \eqref{qs1} reads
\begin{equation}\label{Kgamma0}
\partial_t u + (\s x-mf)(1-x) \partial_x u =0,
\end{equation}
and has  a stable rest point at $\bar x= mf/\s$. Its solution can be explicitly written as
 \[
X(x,t)= \dfrac{mf}{\s} + \dfrac{x-\frac{mf}{\s}}{1+ \frac{\s(1-x)}{\s-mf}(e^{(\s-mf)t}-1)} ,
\]
 and only depends by the parameters $\s$ and $mf$.
In the rare mutation setting, there is an entire curve of parameters $(\gamma, \lambda)\in(0,1)\times(m,\infty)$  that give back the same $m$ and $\s$: this curve can be seen as the graph $\lambda=m/\gamma$. 
As $\gamma$ goes to $0$, the time intensity  $\lambda$ increases, and the paths of the point process driving mutations becomes continuous. On the contrary, at $\gamma=1$ the time intensity gets its minimum $\lambda=m$, and mutations are concentrated in rare events that happen simultaneously to all individuals. The respective Kolmogorov equation is
\begin{equation}\label{Kgamma1}
\partial_t u +\s x (1-x) \partial_x u= mf \left[ u(1,t) -u(x,t)\right] . 
\end{equation}
It is easy to check that the only solution with $u(x,0)=x$ is 
\[
 Z(x,t)=1- \dfrac{(1-x) \, e^{(\s-mf)t}}{1+ (1-x)(e^{\s t}-1)} .
\]
In order to study the dependence of the expected density by the parameter $\gamma$ (equivalently, by the time intensity $\lambda=m/\gamma$), we
denote by $u_{\gamma}$ the respective solution of \eqref{K1}, namely
\begin{equation}\label{Kgamma}
\left\{\begin{array}{ll}
\partial_t u_{\gamma} +\s x (1-x) \partial_x u_{\gamma}= \dfrac{mf}{\gamma} \jcal_0 u_{\gamma} , & 0\le x\le 1, \, t>0 \\
u_{\gamma}(x,0)=x , & 0\le x\le 1.\end{array}\right.
\end{equation}
The graph of $(0,1]\times[0,1]\times[0,+\infty)\ni(\gamma,x,t) \mapsto u_{\gamma}(x,t)$ is a continuous hypersurface that spans the region between the graph of $X(x,t)$ and the one of $Z(x,t)$.

\begin{proposition}\label{gammasurface}
For every $(x,t)$, the function $(0,1]\ni\gamma\mapsto u_{\gamma}(x,t)$ is nondecreasing and continuous, with $u_1(x,t)=Z(x,t)$ and $\lim\limits_{\gamma\to 0} u_{\gamma}(x,t)=X(x,t)$.
Moreover both continuity and convergence are uniform w.r.t.~$(x,t)$ in each compact set $[0,1]\times [0,T]$. 
\end{proposition}
\begin{proof}
To begin with, we check that the functions $u_{\gamma}$ are continuous and ordered w.r.t.~$\gamma$.
A (formal) derivation of equation \eqref{Kgamma} yields that $w(x,t,\gamma)=\partial_{\gamma} u_{\gamma}(x,t)$ solves
\begin{align}\label{dergamma}
\left\{\begin{array}{lr}
\partial_t w +\s x (1-x) \partial_x w = \dfrac{mf}{\gamma} \jcal_0 w + h  & \quad  0\le x\le 1, t>0, 0<\gamma <1 , \\
w(x,0,\gamma)= 0 & \qquad 0\le x\le 1, t=0, 0<\gamma< 1  ,
\end{array}\right.\end{align}
where
\begin{align*}
 h(x,t,\gamma)= & \dfrac{mf}{\gamma^2} \left[ u_{\gamma}(x,t)-u_{\gamma}(x+\gamma(1-x),t)+\gamma(1-x) \partial_xu_{\gamma}(x+\gamma(1-x),t)\right]   \\
= & \dfrac{mf}{2} (1-x)^2 \int_0^1\theta \partial^2_{xx} u_{\gamma}(x+\theta\gamma(1-x),t)d\theta .
\end{align*}
By Lemma \ref{convex}, $0\le h\le \c e^{\mu t}$ (with, possibly, a different constant $\c$). Hence comparison principle gives $0\le w\le \c e^{\mu t} $.
This yields, in turn, that the function $(0,1]\ni\gamma\mapsto u_{\gamma}(x,t)$ is nondecreasing and Lipschitz continuous, and furnishes an estimate of the Lipschitz constant, which is equibounded for all $(x,t) \in[0,1]\times[0,T]$, as $T>0$.
In particular, $u_{\gamma}$ gets near $Z$ as $\gamma\to 1$, with uniform convergence for $(x,t)\in [0,1]\times[0,T]$. \\
We next check that $u_{\gamma}$ approaches $X(x,t)$, as $\gamma\to 0$, by  the viscosity solution approach.
Let us begin by defining
\[ u^+(x,t)=\limsup\limits_{(y,s,\gamma)\to(x,t,0)} u_{\gamma}(y,s) \quad \text{and} \quad u^-(x,t)=\liminf\limits_{(y,s,\gamma)\to(x,t,0)} u_{\gamma}(y,s) .\]
By construction, $u^+$ and $u^-$ are respectively upper and lower semicontinuous, moreover $u^+(x,t)\ge u^-(x,t)$, and certainly $u^+(x,0)=x=u^-(x,0)$. It is trivial to check that $u^+$ and $u^-$ are (possibly discontinuous) viscosity sub and supersolution to the transport equation \eqref{Kgamma0}.
Therefore by comparison $u^+\le u^-$. Thus $u^+(x,t)=u^-(x,t)$ is continuous and equal to $\lim\limits_{\gamma\to 0} u_{\gamma}(x,t)$. Next, uniqueness for the transport equation yields that $\lim\limits_{\gamma\to 0} u_{\gamma}(x,t) =X(x,t)$ pointwise.
Eventually, Dini's monotone convergence Theorem implies uniform convergence on any compact set $[0,1]\times[0,T]$. 
\end{proof}

The family  $u_{\gamma}(x,t)$ spans the segment between $X(x,t)$ and $Z(x,t)$. 
As time increases,  the quasispecies solution $X(x,t)$ converges to the equilibrium point $\bar{x}=mf/\s<1$, while $Z(x,t)\to 1$. Similarly we expect that the asymptotic equilibrium of $u_{\gamma}$ spans the segment between $\bar x$ and 1, as $\gamma$ goes from 0 to 1.
To this aim we investigate the large time behavior of the functions $u_{\gamma}$. Trivially $\lim\limits_{t\to+\infty} u_{\gamma}(1,t)=1$ for any $\gamma$.
Besides $\lim\limits_{t\to+\infty}u_{\gamma}(x,t)$ does not depends by $x\in[0,1)$, for all values of $\gamma$ except at most one.

\begin{proposition}\label{ltbwum}
Take  $\s> mf$ and  let $\gamma^{\ast}\in(0,1)$ be the only solution to 
\[\s\gamma + mf\log(1-\gamma) =0 .\] 
If $\gamma\in(0,1)\setminus\{\gamma^{\ast}\}$, then there exist a number $\bar u_{\gamma}$ so that $\lim\limits_{t\to+\infty}u_{\gamma}(x,t)=\bar u_{\gamma}$ for every $x\in[0,1)$.
\end{proposition}
\begin{proof}
We establish that for every $\gamma\in(0,1)$, $\gamma\neq\gamma^{\ast}$, there exist $\alpha\in(0,2)$ and $\beta>0$ such that
 the solution to \eqref{K1} satisfies 
\begin{equation}\label{u_xestpazzi}
\partial_xu(x,t)\le 2 e^{-\beta t}/\left((1-x)^{\alpha}+e^{-(\s-mf+\beta) t}\right).
\end{equation}
The thesis follows by \eqref{u_xestpazzi} by the same arguments of Proposition \ref{ltb}.
In view of proving \eqref{u_xestpazzi}, we follow the line of Lemma  \ref{u_xestfine} and check that, for any $\alpha \in(0,2)$, there exists $\beta(\alpha)\in\R$ (possibly negative) such that
\[\overline p = 2 e^{ (\s-mf) t}/(1+ (1-x)^{\alpha}e^{ (\s-mf+\beta) t})\] is a supersolution to \eqref{K1_x}.
Set $z=(1-x)^{\alpha}e^{ (\beta+\s-mf) t}$, we have by computations
\begin{align*}
&\partial_t\bar p +  \s (1-x) x \,  \partial_x \bar p +\left(\s(1-2x) +mf\right) \bar p  \\
& = \dfrac{2 e^{ (\s-mf) t}}{(1+z)^2} \left(2\s(1-x) -\s (2-\alpha) x+ \left(\s +mf-\beta  \right)z\right) \\
& \ge \dfrac{2 z \, e^{ (\s-mf) t}}{(1+z)^2}  \left(\s(\alpha-1) + mf - \beta \right) , \\
&\jcal_0\bar p = 
\dfrac{2 z \,e^{ (\s-mf) t}}{(1+z)^2} \kappa(z) ,
\intertext{for $\kappa(z)=\dfrac{(1-(1-\gamma)^{\alpha})(1+z)}{1+(1-\gamma)^{\alpha}z}$. Since $\kappa$ is monotone increasing we get}
&\jcal_0\bar p \le 
\dfrac{2 z \, e^{ (\s-mf) t}}{(1+z)^2}  \left((1-\gamma)^{-\alpha}-1\right)  .
\end{align*}
Hence $\bar p$ is a supersolution provided that $\beta \le \beta(\alpha)= \s(\alpha-1) - mf \dfrac{(1-\gamma)^{1-\alpha}-1}{\gamma}$.
We conclude the proof by showing that $[0,2]\ni\alpha\mapsto \beta(\alpha)$ has a positive maximum.
Indeed, $\beta(\alpha)$ is strictly convex with $\beta(1)=0$, therefore its maximum is positive unless it is reached at $\alpha=1$.
But $\alpha=1$ is not a critical point for $\gamma\neq\gamma^{\ast}$, because $ \beta'(1) = \s +  \log (1-\gamma)\, mf /\gamma \neq 0 $.
\end{proof}

We already know  that $mf/\s\le \bar u_{\gamma}\le 1$; actually we can prove more, namely that $\bar u_{\gamma}\to mf/\s$ as $\gamma\to 0$, and $\bar u_{\gamma}\to 1$ as $\gamma\to 1$.

\begin{proposition}\label{collauno}
We have $\lim\limits_{\gamma\to 1}\bar u_{\gamma}=1$, and $\lim\limits_{\gamma\to 0}\bar u_{\gamma}=\frac{mf}{\s}$.
Indeed, for every $\ep>0$ and $L\in(0,1)$, there are $T>0$ and $\Gamma_0,\Gamma_1\in(0,1)$ so that 
\begin{align*}
1-\ep \le  u_{\gamma}(x,t) & \le 1 \quad & \mbox{for all $\gamma\in[\Gamma_1,1]$, $x\in[0,1]$, $t\ge T$,  }
\\
\bar x \le  u_{\gamma}(x,t) & \le \bar x+\ep \quad & \mbox{for all $\gamma\in(0,\Gamma_0]$, $x\in[0,L]$, $t\ge T$.}
\end{align*}
\end{proposition}
\begin{proof}
We first deal with $\gamma$ near 1. The function $(x,t,\gamma)\mapsto u_{\gamma}(x,t)$ is monotone increasing both w.r.t.~$x$ and $\gamma$. Therefore $u_{\gamma}(x,t) \ge u_{\Gamma}(0,t)$ for all $x\in[0,1]$ and $\gamma\in[\Gamma,1]$. Besides also $t\mapsto u_{\Gamma}(0,t)$ is monotone increasing w.r.t.~$t$ because by \eqref{uformula}
\[ u_{\Gamma}(0,t) = \dfrac{mf}{\Gamma} \int_0^t \left( u_{\Gamma}(\Gamma,s)- u_{\Gamma}(0,s)\right) ds \]
with $ u_{\Gamma}(\Gamma,s)\ge  u_{\Gamma}(0,s)$ for any $s$. 
Hence $u_{\gamma}(x,t) \ge u_{\Gamma}(0,T)$ for all $x\in[0,1]$, $t\ge T$ and $\gamma\in[\Gamma,1]$ and the statement is proved by exhibiting $T$ and $\Gamma$  so that $u_{\Gamma}(0,T) \ge 1-\ep$.
But, since $u_1\to 1$ as $t\to +\infty$, there exists $T$ such that $u_1(0,T)\ge 1-\ep/2$. Next  Proposition \ref{gammasurface} ensures that there is $\Gamma$ so that $u_{\Gamma}(0,T)\ge u_1(0,T)- \ep/2 \ge 1-\ep$ as desired.

Concerning the behavior for small $\gamma$, we may assume without loss of generality that $\bar x<1-\ep $.
We next perturb the mutation coefficient by means of $m_{\ep}=m(1+\ep\s/2f)$, and denote by $X_{\ep}$ the solution of the corresponding quasispecies equation \eqref{Kgamma0}.
It is easily seen that \[X_{\ep}(x,t) \le mf/\s+\ep\] for all $x\in[0,L]$ and $t\ge T$, provided that we chose $T$ sufficiently large. \remove{in such a way that
\[
\dfrac{L-\frac{mf}{\s}-\frac{\ep}{2}}{1+\frac{\s(1-L)}{\s-m f-\ep\s}\left(e^{(\s-mf-\ep\s)T}-1\right)} \le \dfrac{\ep}{2}.
\]}

Thus the thesis follows by comparison, if we  exhibit $\Gamma$ such that $X_{\ep}$ is a supersolution to \eqref{Kgamma} for any $\gamma\in(0,\Gamma]$.
But
\begin{align*}
\partial_t X_{\ep} +\s x(1-x)\partial_x X_{\ep} -\frac{m f}{\gamma}\left[ X_{\ep}(x+\gamma(1-x),t)-X_{\ep}\right] \\
=\dfrac{\ep m\s\left(1+(1 -(1+\frac{2f}{\ep\s})\gamma)z\right)}{2\left(1+z\right)^2\left(1+(1-\gamma)z\right)}
\end{align*}
where we have used the notation $z=\dfrac{\s(1-x)}{\s-m_{\ep}f}\left(e^{(\s-m_{\ep}f)t}-1\right) \ge 0$.
Taking $\Gamma=1/(1+\dfrac{2f}{\ep\s})$ ends the proof.
\end{proof}

Eventually, if the point process driving mutation has small intensity (i.e.~if $\gamma$ is near 1), the relative population density does not tend to the quasispecies equilibrium as $t\to +\infty$: the asymptotically stable strategy according to the quasispecies equation is not asymptotically stable in expectation, according to rare mutation.

\def\cprime{$'$}


\begin{thebibliography}{17}
\providecommand{\natexlab}[1]{#1}
\providecommand{\url}[1]{\texttt{#1}}
\expandafter\ifx\csname urlstyle\endcsname\relax
  \providecommand{\doi}[1]{doi: #1}\else
  \providecommand{\doi}{doi: \begingroup \urlstyle{rm}\Url}\fi

\bibitem[Amadori(2003)]{Ama03}
A.~L. Amadori.
\newblock Nonlinear integro-differential evolution problems arising in option
  pricing: a viscosity solutions approach.
\newblock \emph{Differential Integral Equations}, 16\penalty0 (7):\penalty0
  787--811, 2003.
\newblock ISSN 0893-4983.


\bibitem[Athreya et~al.(1988)Athreya, Kliemann, and Koch]{AKK}
K.~B. Athreya, W.~Kliemann, and G.~Koch.
\newblock On sequential construction of solutions of stochastic differential
  equations with jump terms.
\newblock \emph{Systems Control Lett.}, 10\penalty0 (2):\penalty0 141--146,
  1988.
\newblock ISSN 0167-6911.


\bibitem[Calzolari and Nappo(1996)]{CN}
A.~Calzolari and G.~Nappo.
\newblock \emph{Sulla costruzione di un processo di puro salto}.
\newblock Technical Report, University of Roma ``La Sapienza'', 1996 

\bibitem[Champagnat et~al.(2008)Champagnat, Ferri{\`e}re, and
  M{\'e}l{\'e}ard]{CFM08}
N.~Champagnat, R.~Ferri{\`e}re, and S.~M{\'e}l{\'e}ard.
\newblock From individual stochastic processes to macroscopic models in
  adaptive evolution.
\newblock \emph{Stochastic Models}, 24\penalty0 (sup1):\penalty0 2--44, 2008.
\newblock \doi{10.1080/15326340802437710}.

\bibitem[Dieckmann and Law(1996)]{DL96}
U.~Dieckmann and R.~Law.
\newblock The dynamical theory of coevolution: A derivation from stochastic
  ecological processes.
\newblock \emph{Journal of Mathematical Biology}, 34\penalty0 (5-6):\penalty0
  579--612, 1996.

\bibitem[Eigen and Schuster(1979)]{ESbook79}
M.~Eigen and P.~Schuster.
\newblock \emph{The hypercycle: a principle of natural self-organization}.
\newblock Springer-Verlag, 1979.

\bibitem[Hofbauer and Sigmund(1998)]{SHbook}
J.~Hofbauer and K.~Sigmund.
\newblock \emph{{Evolutionary games and population dynamics.}}
\newblock {Cambridge: Cambridge University Press}, 1998.

\bibitem[Jourdain et~al.(2012)Jourdain, M{\'e}l{\'e}ard, and
  Woyczynski]{JMW12levyflights}
B.~Jourdain, S.~M{\'e}l{\'e}ard, and W.~Woyczynski.
\newblock L{\'e}vy flights in evolutionary ecology.
\newblock \emph{Journal of Mathematical Biology}, 65\penalty0 (4):\penalty0
  677--707, 2012.
\newblock \doi{10.1007/s00285-011-0478-5}.

\bibitem[Lamperti(1977)]{Lamp}
J.~Lamperti.
\newblock \emph{Stochastic processes}.
\newblock New York, 1977.
\newblock A survey of the mathematical theory, Applied Mathematical Sciences,
  Vol. 23.

\bibitem[Maynard~Smith and Price(1973)]{MSP73}
J.~Maynard~Smith and G.~R. Price.
\newblock The logic of animal conflict.
\newblock \emph{Nature}, 246:\penalty0 15--18, 1973.
\newblock \doi{10.1038/246015a0}.

\bibitem[Nowak(2006)]{Nowakbook}
M.~A. Nowak.
\newblock \emph{{Evolutionary dynamics. Exploring the equations of life.}}
\newblock Cambridge, MA: The Belknap Press of Havard University Press, 2006.

\bibitem[Sayah(1991)]{Sayah91}
A.~Sayah.
\newblock {First order Hamilton-Jacobi equations with integro differential
  terms. I: Uniqeness of viscosity solutions, II: Existence of viscosity
  solutions. }.
\newblock \emph{Commun. Partial Differ. Equations}, 16\penalty0 (6-7):\penalty0
  1057--1074 and 1075--1093, 1991.
\newblock \doi{10.1080/03605309108820789}.

\bibitem[Stadler and Schuster(1992)]{SS92mutation}
P.~Stadler and P.~Schuster.
\newblock Mutation in autocatalytic reaction networks.
\newblock \emph{Journal of mathematical biology}, 30\penalty0 (6):\penalty0
  597--632, 1992.

\bibitem[Taylor and Jonker(1978)]{TJ1978}
P.~D. Taylor and L.~B. Jonker.
\newblock Evolutionary stable strategies and game dynamics.
\newblock \emph{Mathematical Biosciences}, 40\penalty0 (1–2):\penalty0 145 --
  156, 1978.
\newblock ISSN 0025-5564.
\newblock \doi{10.1016/0025-5564(78)90077-9}.

\bibitem[Traulsen et~al.(2006)Traulsen, Claussen, and Hauert]{TCH06}
A.~Traulsen, J.~C. Claussen, and C.~Hauert.
\newblock Coevolutionary dynamics in large, but finite populations.
\newblock \emph{Phys. Rev. E}, 74:\penalty0 Article number 011901, Jul 2006.
\newblock \doi{10.1103/PhysRevE.74.011901}.

\end{thebibliography}
\end{document}